\documentclass[12pt,reqno]{amsart}

\usepackage{amsfonts,amsthm,amsmath,microtype}
\usepackage{amsaddr}
\newcommand{\numberset}{\mathbb} 
\newcommand{\R}{\numberset{R}}

\usepackage[a4paper, total={125mm, 185mm}]{geometry}

\usepackage{verbatim}
\usepackage{setspace}
\usepackage[T1]{fontenc}
\usepackage[utf8]{inputenc}
\usepackage{mathptmx}

\makeatletter

\numberwithin{equation}{section}
\newtheorem{thm}{\indent\bf {Theorem}}[section]

\theoremstyle{definition}

\newtheorem{rmk}[thm]{\indent \textsc {Remark}}

\begin{document}

\def\author@andify{%
   \nxandlist {\unskip ,\penalty-1 \space\ignorespaces}%
     {\unskip {} }%
     {\unskip ,\penalty-2 \space }%
}
\title[Improved Hardy inequalities]{Improved Hardy inequalities with a class of weights}

\author[A. Canale]{Anna Canale}
\address{Dipartimento di Matematica,
Università degli Studi di Salerno, \break
Via Giovanni Paolo II, 132, 84084 Fisciano
(Sa), Italy.
}

\thanks{
The author is member of the Gruppo Nazionale per l'Analisi Matematica, 
la Probabilit\'a e le loro Applicazioni 
(GNAMPA) of the Istituto Nazionale di Alta Matematica (INdAM)}


\maketitle

\bigskip

\begin{abstract}

In the paper we state conditions on potentials $V$ to get the improved 
Hardy inequality with weight 
\begin{equation*}
\begin{split}
c_{N,\mu}\int_{\R^N}\frac{\varphi^2}{|x|^2}\mu(x)dx&+
\int_{\R^N}V\,\varphi^2\mu(x)dx
\\&\le 
\int_{\R^N}|\nabla \varphi|^2\mu(x)dx
+C_\mu\int_{\R^N} \varphi^2\mu(x)dx,
\end{split}
\end{equation*} 
for functions $\varphi$ in a weighted Sobolev space 
and for weight functions $\mu$ of a quite general type.
Some local improved Hardy inequalities are also given.
To get the results we use a generalized vector field method.

\end{abstract}

\maketitle

\bigskip


\section{Introduction}

\bigskip

The paper deals with local and non-local improved Hardy inequalities with a class of weights $\mu$ 
of a quite general type and with inverse square potentials perturbed by a function $V$.
The paper fits into the context of Hardy type inequalities with weight stated in \cite{CA Hardy type}. 

The classical Hardy inequality was introduced in 1920th \cite{Hardy20} in the one dimensional case (see also \cite{Hardy25, HLP}). 


The classical Hardy inequality is (see, e.g., \cite{D, KMP, KO} for historical reviews, and \cite{M}).
\begin{equation}\label{iHi introd}
c_{o}(N)\int_{\R^N}\frac{\varphi^2}{|x|^2}dx
\le \int_{\R^N}|\nabla \varphi|^2 dx
\end{equation}
for any functions $\varphi\in H^1(\R^N)$, where $c_{o}(N)=\bigl(\frac{N-2}{2}\bigr)^2$ is the optimal constant. 

Weighted Hardy inequalities with optimal constant depending on weight function $\mu$ have been stated in \cite{GGR, CP} with Gaussian measures
and inverse square potentials with a single pole and in the multipolar case, respectively. 
In a setting of more general measures we refer to \cite{CGRT, CPT1, CPT2}, 
in  the last paper with multipolar potentials.

In particular in \cite{CPT1} the authors proved the inequality

\begin{equation}\label{wHi intro}
c_{N,\mu}\int_{\R^N}\frac{\varphi^2}{|x|^2}\,\mu(x)dx
\le \int_{\R^N}|\nabla \varphi|^2\,\mu(x)dx+
C_\mu \int_{\R^N} \varphi^2\,\mu(x)dx,
\end{equation}
for any functions $\varphi$ in a weighted Sobolev space  
with $c_{N,\mu}=\bigl(\frac{N+K_\mu-2}{2}\bigr)^2$, optimal constant, and $K_\mu$, $C_\mu$ constants depending on $\mu$.  For example when $\mu=\frac{1}{|x|^\gamma}$, $\gamma<
N-2$, it occurs $K_\mu=-\gamma$ and $C_\mu=0$ while if  $\mu= e^{-\delta|x|^2}$, $\delta>0$, we get $K_\mu=0$ (see \cite{CPT1}).

In this paper we improve this results by adding a nonnegative correction term
in the left-hand side in (\ref{wHi intro}).
  
In particular we state sufficient conditions on $V$ to get in $\R^N$ the estimate

\begin{equation}\label{wiHi introd}
c_{N,\mu}\int_{\R^N}\frac{\varphi^2}{|x|^2}\,d\mu+
\int_{\R^N}V\,\varphi^2\,d\mu
\le \int_{\R^N}|\nabla \varphi|^2\,d\mu+
C_\mu \int_{\R^N} \varphi^2\,d\mu,
\end{equation}
 where $d\mu=\mu(x)dx$, for any functions $\varphi$ in a suitable Sobolev space with weight satisfying suitable local integrability assumptions and $c_{N,\mu}$ the constant in (\ref{wHi intro}).

To prove the result we use a method introduced in \cite{CA improved} 
for a class of weights satisfying the H\"older condition.
In this paper we prove the result in the context of more general measures.

Example of weight functions are shown in the paper.
 
In the case of the Lebesgue measure there is a very huge literature on the extension of Hardy's inequality. In particular
the improved version of the classical Hardy 
inequality in bounded domain $\Omega$ in $\R^N$, $N\ge 3$, 
\begin{equation}\label{BV}
\left(\frac {N-2}{2}\right)^2
\int_{\Omega}\frac{\varphi^2}{|x|^2}\,dx+
c_\Omega\int_{\Omega}\,\varphi^2\,dx
\le \int_{\Omega}|\nabla \varphi|^2\,dx
\end{equation}
has been stated in \cite{BV} for all $\varphi \in H^1_0 (\Omega)$.

Later improvements of the estimate (\ref{BV}) of the type

\begin{equation}\label{with V}
\left(\frac {N-2}{2}\right)^2
\int_{\Omega}\frac{\varphi^2}{|x|^2}\,dx+\int_{\Omega}V\,\varphi^2\,dx
\le \int_{\Omega}|\nabla \varphi|^2\,dx
\end{equation}
can be found, for example, in \cite{Adimur, Filippas Ter, Gouss, Musina}.

Reasoning as in the proof of Theorem \ref{wiHi}, 
we deduce weighted versions of the local estimate
(\ref{with V}) for some functions $V$, well-known inequalities when $\mu=1$. In particular we focus our attention on the inequalities 

\begin{equation}\label{log w intro}
\begin{split}
c_{N,\mu}&
\int_{B_1}\frac{\varphi^2}{|x|^2}\,d\mu+\frac{1}{4}\int_{B_1}
\frac{\varphi^2}{|x|^2|\log |x||^2}\,d\mu
\\&\le 
\int_{B_1}|\nabla \varphi|^2\,d\mu+C_\mu\int_{B_1} \varphi^2 \,d\mu
\end{split}
\end{equation}
and, for $\beta\in (0,2]$,
\begin{equation}\label{1- r beta w intro}
c_{N,\mu}
\int_{B_1}\frac{\varphi^2}{|x|^2}\,d\mu+\beta^2\int_{B_1}
\frac{\varphi^2}{|x|^{2-\beta}}\,d\mu
\int_{B_1}|\nabla \varphi|^2\,d\mu+C_\mu\int_{B_1} \varphi^2 \,d\mu
\end{equation}
for any functions $\varphi \in C^\infty_c(B_1)$, 
where $B_1$ is the unit ball in $\R^N$.
For $\beta=2$ we get the weighted version of (\ref{BV}) with $4$ in place of $c_\Omega$.

The Hardy inequalities are applied in many fields.
From a mathematical point of view, a motivation for us to study Hardy inequalities with weights and related improvements is due to 
the applications to evolution problems

$$
(P)\quad \left\{\begin{array}{ll}
\partial_tu(x,t)=Lu(x,t)+{\tilde V}(x)u(x,t),\quad \,x\in {\mathbb R}^N, t>0,\\
u(\cdot ,0)=u_0\geq 0\in L_\mu^2
\end{array}
\right. $$
where $L_\mu^2:=L^2(\R^N, d\mu)$ and $L$ is the 
Kolmogorov operator
\begin{equation}\label{L Kolmogorov}
Lu=\Delta u+\frac{\nabla \mu}{\mu}\cdot\nabla u,
\end{equation}
defined on smooth functions, perturbed by singular potentials ${\tilde V}$. 

An existence result can be obtained, reasoning as \cite{CGRT}, following
Cabr\'e-Martel's approach based on the relation between the weak solution of $(P)$
and the estimate of the {\it bottom of the spectrum} of the operator $-(L+{\tilde V})$
\begin{equation*}
\lambda_1(L+{\tilde V}):=\inf_{\varphi \in H^1_\mu\setminus \{0\}}
\left(\frac{\int_{{\mathbb R}^N}|\nabla \varphi |^2\,d\mu
-\int_{{\mathbb R}^N}{\tilde V}\varphi^2\,d\mu}{\int_{{\mathbb R}^N}\varphi^2\,d\mu}
\right)
\end{equation*}
that results from the Hardy inequality.
In the case $\mu=1$ Cabr\'e and Martel in \cite{CabreMartel}
showed that the boundedness  of 
 $\lambda_1(\Delta+{\tilde V})$ is a necessary 
and sufficient condition for the existence of positive exponentially bounded in time
solutions to the associated initial value problem. 
Later in \cite{GGR, CGRT, CPT2} similar results have been extended to Kolmogorov operators 
perturbed by inverse square potentials, in the last paper in the multipolar case.
The proof uses some properties of the operator $L$ and of its corresponding semigroup
in $L_\mu^2(\R^N)$.  

In the paper we include the existence result for the sake of completeness. 

The paper is organized as follows.

In Section 2 we introduce the class of weights and the conditions on the potentials $V$ 
with some examples.
In Section 3 we state the improved weighted Hardy inequality and some consequences.
Section 4 is devoted to the weighted local estimates.
Finally, in Section 5, we show an application of the estimates to evolution problems.

\bigskip\bigskip

\section{ A class of weight functions and potentials}

\bigskip

Let $\mu\ge 0$ be a weight function on $\R^N$. We define the weighted Sobolev space 
$H^1_\mu=H^1(\R^N, \mu(x)dx)$
as the space of functions in $L^2_\mu:=L^2(\R^N, \mu(x)dx)$ whose weak derivatives belong to
$L_\mu^2$.

The class of function $\mu$  we consider fulfills the conditions
\medskip
\begin{itemize}
\item[$H_1)$] 
\begin{itemize}
\item[$i)$] $\quad \sqrt{\mu}\in H^1_{loc}(\R^N)$;
\item[$ii)$]  $\quad \mu^{-1}\in L_{loc}^1(\R^N)$.
\end{itemize}
\end{itemize}
Let us observe that under the assumption $i)$ $H_1)$ we get
$\mu, \nabla \mu \in L_{loc}^1(\R^N)$.
The reason we suppose $H_1)$ is that we need the density of
the space $C_c^{\infty}(\R^N)$ in $H_{\mu}^1$ (see e.g. \cite{T}). So we can regard
$H_{\mu}^1$ as the completion of $C_c^{\infty}(\R^N)$ with respect to the Sobolev norm
$$
\|\cdot\|_{H^1_\mu}^2 := \|\cdot\|_{L^2_\mu}^2 + \|\nabla \cdot\|_{L^2_\mu}^2.
$$
We introduce, in the proof of the Hardy inequality in the next Section, the function
$f=\frac{g}{|x|^\alpha}$, $g$ radial function, for suitable values of $\alpha$.
We need the following condition on $g$.

\smallskip

\begin{itemize}
\item[$H_2)$] 
\begin{itemize}
\item[$i)$] $\quad   g>0, \quad  \displaystyle-\frac{1}{g}\frac{\partial g}{\partial x_j}
\sqrt{\mu}\in  L^2_{loc}(\R^N),\>
\frac{1}{g}\frac{\partial ^2 g}{\partial x_j^2}\mu\in L^1_{loc}(\R^N)$;
\item[$ii)$] $\quad   \displaystyle-\frac{\Delta g}{g}
+(N-2)\frac{x}{|x|^2}
\cdot \frac{\nabla g}{g}\ge 0$.
\end{itemize}
\end{itemize}

\smallskip

\noindent The assumption on the potential $V$ in the estimates is the following

\smallskip

\begin{itemize}
\item[$H_3)$] $\quad  V=V(x)\in L^1_{loc}(\R^N)$ and
$$      
0\le V\le W:=-\frac{\Delta g}{g}+(N-2)\frac{x}{|x|^2}\cdot \frac{\nabla g}{g}=
-\frac{g''}{g}-\frac{g'}{\rho g},
$$
where $g'$, $g''$ are the first and the second derivatives with respect to $\rho=|x|$, respectively.
\end{itemize}
Under condition $i)$ in $H_2$ we can integrate by parts in the proof of the inequality in the next Section.
The class of radial functions $g$ satisfying $ii)$ in $H_2)$ is such that
\begin{equation}\label{g' r}
(g'\rho)'\le 0
\end{equation}
which implies that $ g'\rho$ is decreasing, so we have
\begin{equation}\label{cond g}
g'(r)\,r\le g'(r_0)\,r_0, \quad\qquad r_0\le r.
\end{equation}
If $g\in C^2(\R^N\setminus\{0\})$ this condition involves
that the function \break
$g(r)-c_1 \log r$ is decreasing
$$
g(r)-c_1 \log r< g(r_0)-c_1 \log r_0, \qquad
c_1=g'(r_0)r_0,
$$
as we can see integrating (\ref{cond g}) in $[r_0, r]$, $r_0> 0$.

Functions satisfying condition (\ref{g' r}) in $(0,R)$, $R<1$, for example,
are the functions $g(r)=|\log r|^\beta$, $\beta\in (0,1)$, $g(r)=1-r^\beta$, $\beta\in (0,2]$. 

The functions $W$ such that there exists a positive radial solution of the Bessel equation 
associated to the potential $W$
\begin{equation}\label{Bessel eq}
g'' +\frac{g'}{r}+W g=0
\end{equation}
are good functions. We observe that, if $W=1$, the Bessel function $J_0$ is a 
positive solution of the equation (\ref {Bessel eq}).

The author in \cite{Gouss} proved that, under suitable hypotheses, the equation
 (\ref{Bessel eq}) is a necessary and sufficient condition to get an
improved Hardy inequality in bounded domains in $\R^N$. 

A further assumption we need is the following.

\smallskip

\begin{itemize}
\item[$H_4)$] $\quad$ There exist constants $K_1, K_2, K_3\in \R$,  $K_2>2-N$
and $K_3= K_2$ if $K_2\ne 0$, $K_3\le0$ if $K_2= 0$, such that
$$
\left(\frac{\nabla g}{g}- \alpha\frac{x}{|x|^2}\right)\cdot \frac{\nabla\mu}{\mu}\le
K_1-\frac{\alpha K_2}{|x|^2}
+K_3\frac{x}{|x|^2}\cdot\frac{\nabla g}{g}
$$
\end{itemize}

\smallskip

\noindent or, equivalently, such that the function $g=g(r)$ satisfies the inequality
\begin{equation}\label{radial condition}
\frac{g'}{ g}\left(\frac{\mu'}{\mu}-\frac{K_3}{r}\right)\le K_1+
\frac{\alpha}{r}\left(\frac{\mu'}{\mu}-\frac{K_2}{r}\right).
\end{equation}
For $g$ fixed, it is a condition for $\mu$.  For example, if $g=1$, 
weight functions satisfying $H_4)$ are the functions
\begin{equation}\label {ex w}
\mu(x)=\frac{1}{|x|^\gamma}e^{-\delta |x|^m}, \quad \delta \ge 0, \quad \gamma<N-2,
\end{equation} 
for suitable values of $m$ (see \cite {CPT1}).
Conversely, for $\mu$ fixed, $H_4)$ represents a condition on $g$.

Finally, we remark that the weights in (\ref{ex w}) fulfill condition $H_1)$.

\bigskip\bigskip

\section{ Weighted improved Hardy inequalities}

\bigskip

In this Section we state a weighted improved Hardy inequality in the setting of more general measure 
with respect to \cite{CA improved}. This allow us to improve the results in \cite {CPT1} 
on weighted Hardy inequalities by adding a nonnegative correction term in the estimates.

The method to get the result has been introduced in \cite{CA improved} for a class of weights 
satisfying the H\"older condition.  We enlarge the class of weights for which we can state the result. 
For this class a weighted Hardy inequality with a different method has been stated in \cite{CPT1}.

The next result states sufficient conditions to get an improved Hardy inequality with weight.

\smallskip

\begin{thm}\label{Thm wiHi}
If $H_1)$--$H_4)$ hold, then we get the estimate
\begin{equation}\label {wiHi}
\begin{split}
\frac{(N+K_2-2)^2}{4}&\int_{\R^N}\frac{\varphi^2}{|x|^2}\,d\mu+
\int_{\R^N}V\,\varphi^2\,d\mu
\\&
\le \int_{\R^N}|\nabla \varphi|^2\,d\mu+
K_1 \int_{\R^N} \varphi^2\,d\mu
\end{split}
\end{equation}
for any functions $\varphi \in H^1_\mu$.
\end{thm}

\begin{proof}
By the density result, we prove (\ref{wiHi}) for any
$\varphi \in C_{c}^{\infty}(\R^N)$.  

We introduce the vector-valued function 
$$
F=-\frac{\nabla f}{ f}\mu=-\frac{\nabla g}{ g}\mu+\alpha\frac{ x}{|x|^2}\mu,
$$
where $f=\frac{g}{|x|^\alpha}$, $\alpha\in (0, N+K_2-2)$.

We get
$$
{\rm div}F=
-\frac{\Delta f}{f}\mu+\left|\frac{\nabla f}{f}\right|^2\mu
-\frac{\nabla f}{f}\cdot \nabla\mu,
$$
where
\begin{equation*}
\begin{split}
-\frac{\Delta f}{f}&=-|x|^\alpha
\Delta \frac{1}{|x|^\alpha}- 2|x|^\alpha\nabla \frac{1}{|x|^\alpha}\cdot
\frac{\nabla g}{ g}-\frac{\Delta g}{g}
\\&=
\frac{\alpha(N-2-\alpha)}{|x|^2}
+2\alpha\frac{x}{|x|^2}\cdot \frac{\nabla g}{g}-\frac{\Delta g}{g}.
\end{split}
\end{equation*}
Now we observe that $F_j$, $\displaystyle\frac{\partial F_j}{\partial x_j}\in L_{loc}^1(\R^N)$, 
where $F_j$ is the j-th component of $F$. 
Indeed, for any $K$ compact set in $\R^N$, by the H\"older and the classical Hardy inequalities, 
using hypotheses $i)$ in $H_1)$ on $\mu$ and $i)$ in $H_2)$ on $g$, 
we obtain the following estimate

\begin{equation*}
\begin{split}
\int_K |F_j| \,dx&\le
\int_K \left|-\frac{1}{g}\frac{\partial g}{\partial x_j}\right|\mu(x)\,dx
+\alpha\int_K \frac{\mu(x)}{|x|} \,dx
\\&\le
\int_K \left|-\frac{1}{g}\frac{\partial g}{\partial x_j}\right|\mu(x)\,dx
+\alpha\left(\int_K \frac{\mu(x)}{|x|^2}\,dx \right)^{\frac{1}{2}}
\left(\int_K \mu(x)\,dx\right)^{\frac{1}{2}}
\\&\le
\left(\int_K  \left|-\frac{1}{g}\frac{\partial g}{\partial x_j}\right|^2\mu(x)\,dx \right)^{\frac{1}{2}}
\left(\int_K  \mu(x) \,dx\right)^{\frac{1}{2}}
\\&+
\frac{2\alpha} {(N-2)}\left(\int_K \left|\nabla\sqrt \mu\right|^2\,dx \right)^{\frac{1}{2}}
\left(\int_K  \mu(x) \,dx\right)^{\frac{1}{2}}.
\end{split}
\end{equation*}
To obtain the local integrability of the partial derivative of $F_j$ 
\begin{equation}\label{partial derivative}
\begin{split}
\frac{\partial F_j}{\partial x_j}& =\frac{\partial}{\partial x_j}\left(
-\frac{1}{g}\frac{\partial g}{\partial x_j}\mu+\alpha \frac{x_j}{|x|^2}\mu\right)=
\frac{1}{g^2}\left(\frac{\partial g}{\partial x_j}\right)^2\mu
-\frac{1}{g}\frac{\partial^2 g}{\partial x_j^2}\mu
\\&-
\frac{1}{g}\frac{\partial g}{\partial x_j}\frac{\partial \mu}{\partial x_j}+
\alpha\frac{\mu}{|x|^2}-2\alpha \frac{x_j^2}{|x|^4}\mu+
\alpha \frac{x_j}{|x|^2}\frac{\partial \mu}{\partial x_j}
\\&=
d_1+d_2+d_3+d_4+d_5+d_6,
\end{split}
\end{equation}
we estimate the terms on on the right-hand side in (\ref{partial derivative}).
The terms $d_1$, $d_2$ belong to $L_{loc}^1(\mathbb{R}^N)$ by hypotheses, 
\begin{equation*}
\begin{split}
\int_K |d_3|\,dx&\le \left(\int_K \left|-\frac{1}{g}\frac{\partial g}{\partial x_j}\right|^2\mu(x)\,dx
\right)^{\frac{1}{2}}
 \left(\int_K \left|\frac{1}{{\sqrt \mu}}\frac{\partial \mu}{\partial x_j}\right|^2\,dx
\right)^{\frac{1}{2}}
\\&\le
4\left(\int_K \left|
-\frac{1}{g}\frac{\partial g}{\partial x_j}\right|^2\,d\mu
\right)^{\frac{1}{2}}
\left(\int_K \left|{\nabla{\sqrt \mu}}\right|^2\,dx\right)^{\frac{1}{2}},
\end{split}
\end{equation*}
$d_4$, $d_5$ can be estimated using the Hardy inequality and the hypothesis $i)$ in $H_1)$
as above. As regards the remaing term we have 
\begin{equation*}
\begin{split}
\int_K |d_6|\,dx&\le
\alpha \int_K \frac{\sqrt \mu}{|x|}\frac{1}{\sqrt \mu}\left|\frac{\partial \mu}{\partial x_j}\right|\,dx
\\&\le
2\alpha\left(\int_K \frac{\mu}{|x|^2}\,dx\right)^{\frac{1}{2}}
\left(\int_K \left|{\nabla{\sqrt \mu}}\right|^2\,dx\right)^{\frac{1}{2}}.
\\&\le
\frac{4\alpha} {(N-2)}\left(\int_K \left|\nabla\sqrt \mu\right|^2\,dx \right).
\end{split}
\end{equation*}


\noindent Now we start from the following integral

\begin{equation}\label{div F}
\begin{split}
\int_{\R^N}{\rm div}F \,\varphi^2 dx&=
\int_{\R^N}\Biggl[\frac{\alpha(N-2-\alpha)}{|x|^2}+
2\alpha\frac{x}{|x|^2}\cdot \frac{\nabla g}{g}
\\&
-\frac{\Delta g}{g}\Biggr]\varphi^2\,d\mu
+\int_{\R^N}\left|\frac{\nabla g}{ g}-\alpha\frac{x}{|x|^2}\right|^2\varphi^2\,d\mu
\\&-
\int_{\R^N}\left(\frac{\nabla g}{g}-\alpha\frac{x}{|x|^2}\right)\cdot
\frac{\nabla\mu}{\mu}\Biggr]\varphi^2\,d\mu.
\end{split}
\end{equation}
The first step is to estimate the integral on the left-hand side in (\ref{div F}) from above. 
To this aim we integrate by parts and use H\"older's and Young's inequalities to get
\begin{equation}\label{div F right w}
\begin{split}
\int_{\R^N}&{\rm div}F \,\varphi^2 \,d\mu=
-2\int_{\R^N}\varphi F\cdot\nabla\varphi \, d\mu
\\&
\le 2\left(\int_{\R^N}|\nabla \varphi|^2 \, d\mu\right)^{\frac{1}{2}}
\left(\int_{\R^N}\left|\frac{\nabla f}{f}\right|^2\,\varphi^2\, d\mu\right)^{\frac{1}{2}}
\\&
\le \int_{\R^N}|\nabla \varphi|^2 \, d\mu+
\int_{\R^N}\left|\frac{\nabla f}{f}\right|^2 \,\varphi^2\, d\mu
\\&
=\int_{\R^N}|\nabla \varphi|^2 \, d\mu+
\int_{\R^N}\left|\frac{\nabla g}{ g}-\alpha\frac{x}{|x|^2}\right|^2\,\varphi^2\, d\mu.
\end{split}
\end{equation}

On the other hand, starting from (\ref{div F}), by the condition $H_4)$ we obtain  

\begin{equation}\label{div F left w}
\begin{split}
\int_{\R^N}{\rm div}F \,\varphi^2 \,d\mu&\ge
\alpha(N-2-\alpha)\int_{\R^N}\frac{\varphi^2}{|x|^2}\,d\mu
\\&+
\int_{\R^N}\left(2\alpha\frac{x}{|x|^2}\cdot \frac{\nabla g}{g}-\frac{\Delta g}{g}\right)\varphi^2\,d\mu
\\&
+\int_{\R^N}\left|\frac{\nabla g}{ g}-\alpha\frac{x}{|x|^2}\right|^2\varphi^2\,d\mu-
K_1\int_{\R^N}\varphi^2\,d\mu
\\&
+\alpha K_2\int_{\R^N}\frac{\varphi^2}{|x|^2}\,d\mu
- K_3\int_{\R^N} \frac{x}{|x|^2}\cdot \frac{\nabla g}{g}
\varphi^2\,d\mu
\\&=
\alpha(N+K_2-2-\alpha)\int_{\R^N}\frac{\varphi^2}{|x|^2}\,d\mu
\\&
+\int_{\R^N}\left[(2\alpha-K_3)\frac{x}{|x|^2}\cdot \frac{\nabla g}{g}
-\frac{\Delta g}{g}\right]\varphi^2\,d\mu
\\&+
\int_{\R^N}\left|\frac{\nabla g}{ g}-\alpha\frac{x}{|x|^2}\right|^2\varphi^2\,d\mu-
K_1\int_{\R^N}\varphi^2\,d\mu.
\end{split}
\end{equation}
The inequalities (\ref{div F right w}) and (\ref{div F left w}) led us to the estimate 
\begin{equation}\label{eq. with alfa}
\begin{split}
\alpha(N+&K_2-2-\alpha)\int_{\R^N}\frac{\varphi^2}{|x|^2}\,d\mu+
\int_{\R^N} \Biggl[(2\alpha-K_3)\frac{x}{|x|^2}\cdot \frac{\nabla g}{g}
\\&
-\frac{\Delta g}{g}\Biggr]\varphi^2\,d\mu\le 
\int_{\R^N}|\nabla \varphi|^2 \, d\mu+K_1\int_{\R^N}\varphi^2\,d\mu.
\end{split}
\end{equation}
The maximum value of the first constant on the left-hand side in (\ref{eq. with alfa}) is
$$
\max_\alpha \alpha(N+k_2-2-\alpha)=\frac{(N+k_2-2)^2}{4 },
$$
attained for $\alpha=\alpha_o=\frac{(N+k_2-2)}{2 }$.

Observing that $2\alpha_o-K_3\ge N-2$ and taking in mind the condition $H_3)$,
we obtain the inequality  (\ref{wiHi}).

\end{proof}

\smallskip

\begin{rmk}

For $g=1$ and, then, $V=W=0$, we obtain a weighted Hardy inequality. 
For $g=1$, $\mu=1$ and, so, 
when $K_2=0$, the method to get the result in the Theorem \ref{Thm wiHi}
results to be the vector field method used in \cite{M} to prove the Hardy inequality.

\end{rmk}

\smallskip

An example of weight satisfying condition $H_1)$ is the function  \break
$\mu=\frac{1}{|x|^\gamma}$,
for $\gamma<N-2$. In this case the condition (\ref{radial condition})
 is verified for $K_2,K_3\le-\gamma$
for any $K_1 \ge 0$. Then, as a consequence of Theorem \ref{Thm wiHi}, 
we get the inequality 
\begin{equation}\label{log w Rn}
\begin{split}
\frac{(N-\gamma-2)^2}{4}&\int_{\R^N}\frac{\varphi^2}{|x|^2}|x|^{-\gamma}\,dx+
\int_{\R^N}V\,\varphi^2|x|^{-\gamma}\,dx
\\&
\le \int_{\R^N}|\nabla \varphi|^2|x|^{-\gamma}\,dx
\end{split}
\end{equation}
for any functions $\varphi \in H^1_\mu$. For $V=W=0$ the inequality above is the
Caffarelli- Niremberg inequality.

We remark that,
as a consequence of Theorem \ref{Thm wiHi}, we deduce the estimate 
$$
\|{\tilde V}_\mu^{\frac{1}{2}}\varphi\|_{L_\mu^2(\R^N)}\le c \|\varphi\|_{H^1_\mu(\R^N)},
$$
where ${\tilde V} =\frac{(N-2)^2}{4}\frac{1}{|x|^2}+V$ and 
$c$ is a constant independent of $V$ and $\varphi$.

For $L^p$ estimates and embedding results of this type 
with some applications to elliptic equations see
\cite{CA JMAA, CA JIM 2007, CA JIM 2008, CA MIA, CA embedding, CTar}.

Finally, as a direct consequence of the Theorem \ref{Thm wiHi}, we deduce 
the following result concerning a class of general weighted Hardy inequalities
for $V$ satisfying $H_3)$

\begin{equation*}
\int_{\R^N}V\,\varphi^2\,d\mu\le
\int_{\R^N}|\nabla \varphi|^2\,d\mu+
K_1 \int_{\R^N} \varphi^2\,d\mu
\end{equation*}
for any functions $\varphi \in H^1_\mu$.

\bigskip\bigskip

\section{Local estimates}

\bigskip

In this Section we state some local weighted estimates by means of
the metod used to prove Theorem \ref{Thm wiHi}. 
These estimates represent the weighted version of well-known improved Hardy inequalities in bounded subset of $\R^N$ (see \cite{Adimur, Filippas Ter, Gouss, Musina})
and are based on examples of functions $g$ satisfying locally 
the assumptions of the Theorem \ref{Thm wiHi}. 

The first result is the following.

\smallskip 

\begin{thm}\label{Thm wiHi log}
Let $N\ge 3$ and let $B_1$ the unit ball in $\R^N$. Then, under assumptions $H_1)$, $i)$ 
in $H_2)$ and $H_4)$ on $\mu$, we get
\begin{equation}\label{log w}
\begin{split}
\frac{(N+K_2-2)^2}{4}\int_{B_1}&\frac{\varphi^2}{|x|^2}\,d\mu+\frac{1}{4}\int_{B_1}
\frac{\varphi^2}{|x|^2|\log |x||^2}\,d\mu
\\&\le 
\int_{B_1}|\nabla \varphi|^2\,d\mu+K_1\int_{B_1} \varphi^2 \,d\mu
\end{split}
\end{equation}
for any functions $\varphi \in C^\infty_c(B_1)$.
\end{thm}

\begin{proof}
Reasoning as in the proof of Theorem \ref{Thm wiHi}, we set 
$g=|\log|x||^\beta$, $\beta\in (0,1)$. 
Then the function $W$ in $H_3)$ is given by
$$
W=\frac{\beta(1-\beta)}{|x|^2|\log |x||^2}.
$$
To get the integrability required in $i)$ in $H_2)$, 
it is sufficient that the weights $\mu$ are such that $\displaystyle\frac{g'^2}{g^2} \mu$,
$\displaystyle\frac{g''}{g} \mu\in L^1_{loc}(\R^N)$.
More precisely, pointing out that
\begin{equation}\label {derivative g}
-\frac{\partial g}{\partial x_j}=-\frac{x_j}{|x|}g',
\qquad
\frac{\partial ^2 g}{\partial x_j^2}=\frac{g'}{|x|}-\frac{x_j^2}{|x|^2}g'+\frac{x_j^2}{|x|^2}g'',
\end{equation}
if $K$ is a compact set in $B_1$, we get

\begin{equation*}
\int_K \left|-\frac{1}{g}\frac{\partial g}{\partial x_j}\right|\,d\mu\le
\left(\int_K \left|-\frac{1}{g}\frac{\partial g}{\partial x_j}\right|^2\,d\mu
\right)^{\frac{1}{2}}
\left(\int_K \mu(x)\,d\mu \right)^{\frac{1}{2}},
\end{equation*}

\begin{equation*}
\int_K \left|-\frac{1}{g}\frac{\partial g}{\partial x_j}\right|^2\,d\mu\le
\int_K \left|\frac{g'}{g}\right|^2 d\mu=
\int_K \frac{\beta^2}{|x|^{2}|\log |x| |^2}\,d\mu,
\end{equation*}

\begin{equation*}
\begin{split}
\int_K &\frac{1}{|x|}\left|\frac{g'}{g}\right|\,d\mu\le
\left(\int_K  \frac{\mu}{|x|^2}\,dx\right)^{\frac{1}{2}}
\left(\int_K \left|\frac{g'}{g}\right|^2\,d\mu\right)^{\frac{1}{2}}
\\&\le
\frac{2}{N-2}\left(\int_K  |\nabla{\sqrt\mu}|^2\,d\mu\right)^{\frac{1}{2}}
\left(\int_K \frac{\beta^2}{|x|^2|\log |x| |^2}\,d\mu\right)^{\frac{1}{2}}
\end{split}
\end{equation*}
and, about the last term on the right-hand side in (\ref {derivative g}),

\begin{equation*}
\int_K \frac{x_j^2}{|x|^2}\left|\frac {g''}{g}\right|\,d\mu\le
\int_K \left|\frac{g''}{g}\right|\,d\mu=
\int_K \left|\frac{\beta(\beta-1) +\beta|\log |x||}{|x|^2|\log |x| |^2}\right|\,d\mu.
\end{equation*}
Finally, since
$$
\max_{\beta\in(0,1)}[\beta(1-\beta)]=\frac{1}{4},
$$
attained for $\beta=\frac{1}{2}$, we get the result.

\end{proof}

In the case of weight $\mu=\frac{1}{|x|^\gamma}$, $\gamma<N-2$, the inequality 
(\ref{log w}), $K_1=0$ and $K_2=-\gamma$, is the local version of (\ref{log w Rn}) 
with $V=\frac{1}{4}\frac{1}{|x|^2|\log |x||^2}$.

Another example of weight is given by $\mu=\frac{1}{|x|^\gamma} e^{-\delta|x|^m}$,
 $\gamma<N-2$, $\delta, m>0$. In the last case the condition (\ref{radial condition}) in $B_1$ 
is satisfied for $K_2,K_3\le -\gamma-\delta m$, $K_1 \ge 0$.

For $\mu=1$ the inequality (\ref{log w}) results to be the improved Hardy inequality with 
Lebesgue measure in \cite {Adimur, Filippas Ter, Musina}.

 A further local inequality follows.

\smallskip

\begin{thm}\label{Thm wiHi 1-r beta}
Let $N\ge 3$ and let $B_1$ the unit ball in $\R^N$. Then, under assumptions $H_1)$, $i)$ 
in $H_2)$ and $H_4)$ on $\mu$, we get
\begin{equation}\label{1- r beta w}
\begin{split}
\frac{(N+K_2-2)^2}{4}\int_{B_1}&\frac{\varphi^2}{|x|^2}\,d\mu+\beta^2\int_{B_1}
\frac{\varphi^2}{|x|^{2-\beta}}\,d\mu
\\&\le 
\int_{B_1}|\nabla \varphi|^2\,d\mu+K_1\int_{B_1} \varphi^2 \,d\mu
\end{split}
\end{equation}
for any functions $\varphi \in C^\infty_c(B_1)$ and $\beta\in(0, 2]$.
\end{thm}

\begin{proof}

It is enough to consider $g=1-|x|^\beta$, $\beta\in(0, 2]$, observing that
$$
V=\frac{\beta^2}{|x|^{2-\beta}}\le
W=\frac{\beta^2}{|x|^{2-\beta}(1-|x|^\beta)}.
$$

\end{proof}
Finally, we remark that for $\beta=2$ and $\mu=1$ we get almost the estimate
 in \cite{BV, Gouss} in the sense that, in place of 4 in the left-hand side in (\ref{1- r beta w}), 
the authors obtained $z_0^2$, where $z_0$ is the first zero of the Bessel function $J_0(z)$.

Also in this case the functions $\mu=\frac{1}{|x|^\gamma}$ and $\mu=
\frac{1}{|x|^\gamma} e^{-\delta|x|^m}$ are good weights.

\bigskip\bigskip

\section{An application to evolution problems}

\bigskip

In the Section we give a motivation for our interest in Hardy inequalities with weight.
These estimates play a crucial role in achieving existence results for solutions to the problem
$$
(P)\quad \left\{\begin{array}{ll}
\partial_tu(x,t)=Lu(x,t)+{\tilde V(x)}u(x,t),\quad \,x\in {\mathbb R}^N, t>0,\\
u(\cdot ,0)=u_0\geq 0\in L_\mu^2,
\end{array}
\right. $$
where $L$ is the Kolmogorov operator 

\begin{equation*}
Lu=\Delta u+\frac{\nabla \mu}{\mu}\cdot\nabla u
\end{equation*}
defined on smooth functions, perturbed by a potential ${\tilde V(x)}$,
with ${\tilde V(x)}$ sum of an inverse square potential 
and $V$ satisfying condition $H_3)$.

We say that $u$ is a weak solution to ($P$) if, for each $T, R>0 $, we have
$$u\in C(\left[ 0, T \right] , L^2_\mu ), \quad Vu\in L^1(B_R \times \left( 0,T\right) , d\mu dt )$$
and
\begin{equation*}
\int_0^T \int_{\R^N}u(-\partial_t\phi - L\phi )\,d\mu dt -\int_{\R^N}u_0\phi(\cdot ,0)\,d\mu =
\int_0^T \int_{\R^N} Vu\phi \, d\mu dt
\end{equation*}
for all $\phi \in W_2^{2,1}(\R^N \times \left[ 0,T\right])$ having compact support with $\phi(\cdot , T)=0$, 
where $B_R$ denotes the open ball of $\R^N$ of radius $R$ centered at $0$.
For any $\Omega\subset \R^N$, $ W_2^{2,1}(\Omega\times (0,T)) $ is the parabolic Sobolev space 
of the functions $u\in L^2(\Omega \times (0,T)) $ having weak space derivatives $D_x^{\alpha}
u\in L^2(\Omega \times (0,T))$ for $|\alpha |\le 2$ and weak time derivative
$\partial_t u \in L^2(\Omega \times (0,T))$ equipped with the norm 
\begin{equation*}
\begin{split}
\|u\|_{W_2^{2,1}(\Omega\times (0,T))}&:= \Biggl( 
\|u\|_{L^2(\Omega \times (0,T))}^2 + \|\partial_t u\|_{L^2(\Omega \times (0,T))}^2 
\\&+
 \sum_{1\le |\alpha |\le 2} \|D^{\alpha}u\|_{L^2(\Omega \times (0,T))}^2
\Biggr)^{\frac{1}{2}}.
\end{split}
\end{equation*}
An additional assumptions on $\mu$ allows us to get semigroup generation 
on $L^2_\mu$ (see \cite[Corollary 3.7]{alb-lor-man}).

\begin{itemize}
\item [$H_5)$] $\quad\mu \in C_{loc}^{1,\lambda}(\R^N\setminus \{0\})$, $\lambda\in(0,1)$, 
$\mu\in H^{1}_{loc}(\R^N)$,  $\frac{\nabla \mu}{\mu}\in L^r_{loc}(\R^N)$ 
for some $r>N$, and $\inf_{x\in K}\mu (x)>0$ for any compact set $K\subset \R^N$
\end{itemize}
We remark that the condition $H_5)$ implies
 $i)$ in $H_3)$. Indeed if $\mu\in H^{1}_{loc}(\R^N)$
then $\mu\in L^{1}_{loc}(\R^N)$ and $\nabla\mu\in L^{2}_{loc}(\R^N)$. Moreover
$\frac{\nabla \mu}{\mu}\in L^2_{loc}(\R^N)$ since $r>2$. So we get
$$
\int_K  \left|\nabla\sqrt\mu\right|^2\,dx=\frac{1}{4}\int_K \frac{|\nabla\mu|^2}{\mu}\,dx\le
\frac{1}{4}\left(\int_K \left|\frac{\nabla\mu}{\mu}\right|^2\,dx\right)^{\frac{1}{2}}
\left(\int_K \left|\nabla\mu\right|^2\,dx\right)^{\frac{1}{2}}.
$$
An example of weight function satisfying $H_5)$ is $\mu=e^{-\delta |x|^m}$, $\delta, m>0$.

In the applications to evolution problems with Kolmogorov operators we need $C_0$-semigroup generation results reasoning as in \cite{GGR, CGRT, CPT2}.
Operators of a more general type for which the generation of semigroup was stated can be found, for example, in \cite{CMR} in the context of weighted spaces.

The bottom of the spectrum of $-(L+{\tilde V})$ is defined as follows
\begin{equation*}
\lambda_1(L+{\tilde V}):=\inf_{\varphi \in H^1_\mu\setminus \{0\}}
\left(\frac{\int_{{\mathbb R}^N}|\nabla \varphi |^2\,d\mu
-\int_{{\mathbb R}^N}{\tilde V}\varphi^2\,d\mu}{\int_{{\mathbb R}^N}\varphi^2\,d\mu}
\right).
\end{equation*}
 The authors in \cite{CGRT} stated the following result with a proof similar to the one given in 
\cite{CabreMartel}.
We include the hypothesis $ii)$ in $H_1)$ to get the density result.

\smallskip

\begin{thm}\label{theor as CM}

Assume that the $\mu$ satisfies $ii)$ in $H_2)$ and $H_5)$.
Let $0\le  {\tilde V}(x)\in L^1_{loc}(\R^N)$. 
Then, if $\lambda_1(L+{\tilde V})>-\infty$, there exists a
positive weak solution $u\in C([0,\infty),L^2_\mu)$ of $(P)$ satisfying
the estimate
\begin{equation}\label{eq 1}
\|u(t)\|_{L^2_\mu}\le Me^{\omega t}\|u_0\|_{L^2_\mu},\quad t\ge0
\end{equation}
for some constants $M\ge 1$ and $\omega \in {\mathbb R}$. 
 \end{thm}

\smallskip

The existence result below relies on the Theorem \ref{Thm wiHi}
and on the Theorem \ref{theor as CM}.

\smallskip

\begin{thm}
Assume hypotheses $ii)$ in $H_1)$, $H_2)$-- $H_5)$.
Then there exists a positive weak solution $u\in C([0,\infty),L^2_\mu)$ of $(P)$
satisfying
\begin{equation*}
\|u(t)\|_{L^2_\mu}\le Me^{\omega t}\|u_0\|_{L^2_\mu},\quad t\ge0
\end{equation*}
for some constants $M\ge 1$, $\omega \in \R $.
\end{thm}

\begin{proof}
The weighted Hardy inequality (\ref{wiHi}) implies that
$\lambda_1(L+ {\tilde V})>-\infty$. Then the result is a consequence 
of the Theorem \ref{theor as CM}.
\end{proof}

\bigskip\bigskip

{\it email address}: $\>$ acanale@unisa.it


\bigskip\bigskip



\begin{thebibliography}{22}

\bibitem{Adimur}
Adimurthi, N. Chaudhuri, M. Ramaswamy, \textit{An improved Hardy-Sobolev inequality
and its applications}, Proc. Amer. Math. Soc. 130 (2) (2002), pp. 489--505.

\bibitem{alb-lor-man}
A. Albanese, L. Lorenzi, E. Mangino, \textit{$L^p$–uniqueness for elliptic operators with unbounded 
coefficients in $\R^N$}, J. Funct. Anal.  256 (4) (2009), pp. 1238--1257.

\bibitem{BV}
H. Brezis, J.L. Vazquez, \textit{Blow-up solutions  of some nonlinear elliptic problems},
 Rev. Mat. Complut. Madrid  10 (2) (1997), pp. 443--469.

\bibitem{CabreMartel}
X. Cabr\'e, Y. Martel, \textit{Existence versus explosion instantan\'ee pour 
des e\'quations de la chaleur line\'aires 
avec potentiel singulier}, C. R. Acad. Sci. Paris 329 (11) (1999), pp. 973--978.

\bibitem{CA JMAA}
A. Canale,
\textit{A priori bounds in weighted spaces},
J. Math. Anal. Appl.  287 (2003), pp. 101--117. 

\bibitem{CA JIM 2007}
A. Canale,
\textit{On some results in weighted spaces under Cordes type conditions},
J. Interdiscip. Math. 10 (2007), pp. 245--261.

\bibitem{CA JIM 2008}
A. Canale,
\textit{$L^\infty$ estimates for variational solutions of boundary value problems in unbounded domains},
J. Interdiscip. Math. 11 (2008), pp. 127--139. 

\bibitem{CA MIA}
A. Canale,
\textit{Bounds in spaces of Morrey under Chicco type conditions},
Math. Inequal. Appl. 12 (2009), pp. 265--278. 

\bibitem{CA embedding}
A. Canale,
\textit{An embedding result},
J. Interdiscip. Math.  17 (2014), pp. 199--206. 

\bibitem{CA Hardy type}
A. Canale,
\textit{A class of weighted Hardy type inequalities in $\R^N$},
Ric. Mat. (2021), DOI: 10.1007/s11587-021-00628-7. 

\bibitem{CA improved}
A. Canale,
\textit{Local and non-local improved Hardy inequalities with weights},
Accad. Naz. Lincei Rend. Lincei Mat. Appl.
33 (2022), n. 2, pp. 385-398. 

\bibitem{CGRT}
A. Canale, F. Gregorio, A. Rhandi, C. Tacelli, \textit{Weighted Hardy's inequalities and Kolmogorov-type operators}, 
Appl. Anal. 98 (7) (2019), pp. 1236--1254. 


\bibitem{CMR}
A. Canale, R. M. Mininni, A. Rhandi, \textit{Analytic approach to solve a degenerate parabolic PDE  
for the Heston model}, Math. Meth. Appl. Sci. 40 (2017), pp. 4982--4992.


\bibitem{CP}
A. Canale, F. Pappalardo, \textit{Weighted Hardy inequalities and Ornstein-Uhlenbeck type operators perturbed 
by multipolar inverse square potentials}, J. Math. Anal. Appl. 463 (2018), pp. 895--909.

\bibitem{CPT1}
A. Canale, F. Pappalardo, C. Tarantino, \textit{A class of weighted Hardy inequalities and applications 
to evolution problems}, Ann. Mat. Pura Appl. 199 (2020), pp.1171-1181. 

\bibitem{CPT2}
A. Canale, F. Pappalardo, C. Tarantino, \textit{Weighted multipolar Hardy inequalities and evolution problems 
with Kolmogorov operators perturbed by singular potentials}, 
Commun. Pure Appl. Anal. 20 (2021), pp. 405--425. 

\bibitem{CTar}
A. Canale, C. Tarantino,
\textit{Morrey type spaces and multiplication operators in Sobolev spaces},
 arXiv: 1412.6778v1, (21 dicembre 2014). 

\bibitem{D}
 E.B. Davies, \textit{A review of Hardy inequalities}, in "The Maz’ya Anniversary Collection", 
 vol. 2 (Rostock, 1998), in  "Oper. Theory Adv. Appl.", 110, Birkhäuser, Basel, (1999), 55--67. 

\bibitem{Filippas Ter}
S. Filippas, A. Tertikas, \textit{Optimizing improved Hardy inequalities}, J. Funct. Anal. 192  (2002), pp. 186--233.  

\bibitem{Gouss}
N. Ghoussoub, A. Moradif, \textit{On the best possible remaing term in the
Hardy inequality}, Pro. Natl. Acad. Sci. USA  105 (37) (2008), pp. 13746--13751.

\bibitem{GGR}
G. R. Goldstein, J. A. Goldstein, A. Rhandi, \textit{Weighted Hardy's inequality and 
the Kolmogorov equation perturbed by an inverse-square potential}, 
Appl. Anal. 91 (11) (2012), pp. 2057--2071.

\bibitem{Hardy20}
G. H. Hardy,  \textit{Note on a theorem of Hilbert},
Math. Z. 6 n.3-4 (1920),  314--317.

\bibitem{Hardy25}
G. H. Hardy,  \textit{Notes on some points in the integral calculus LX: An inequality between integrals}, Messenger Math., 54 (1925), 150--156. 

\bibitem{HLP}
G. H. Hardy, J. E. Littlewood and G. P´olya, "Inequalities", Reprint of the 1952 edition. Cambridge Mathematical Library,  Cambridge University Press, 1988. 

\bibitem{KMP}
A. Kufner, L. Maligranda, L. Persson, 
"The Hardy Inequality: About Its History and Some Related Results", Vydavatelsý Servis, Plzen, 2007.

\bibitem{KO}
 A. Kufner, B. Opic, "Hardy-Type Inequalities", Pitman Research Notes in Math., vol. 219, Longman, Harlow, 1990. 

\bibitem{M}
E. Mitidieri, {\it{A simple approach to Hardy inequalities}}, Math. Notes 67 (4) 
(2000), pp. 479-486.

\bibitem{Musina}
R. Musina, \textit{A note on the paper "Optimizing improved Hardy inequalities" by 
S. Filippas and A. Tertikas}, J. Funct. Anal. 256 (2009), pp. 2741--2745.

\bibitem{T}
J. M. T\"olle, 
\textit{Uniqueness of weighted Sobolev spaces with weakly differentiable weights}, 
J. Funct. Anal. 263 (2012), pp. 3195--3223.

\end{thebibliography}
\end{document}